\documentclass[11pt]{amsart}
\usepackage[margin=1.15in]{geometry}
\usepackage{amsmath,amssymb,amsthm,mathtools}
\usepackage{booktabs,enumitem,graphicx}
\usepackage[colorlinks=true,linkcolor=blue,citecolor=blue,urlcolor=blue]{hyperref}
\newtheorem{theorem}{Theorem}[section]
\newtheorem{lemma}[theorem]{Lemma}
\newtheorem{proposition}[theorem]{Proposition}
\newtheorem{corollary}[theorem]{Corollary}
\newtheorem*{conjecture}{Conjecture}
\newtheorem{mainthm}{Theorem}

\theoremstyle{definition}

\theoremstyle{remark}
\newtheorem{remark}[theorem]{Remark}

\newcommand{\E}{\mathbb E}
\newcommand{\Pp}{\mathbb P}
\newcommand{\Var}{\operatorname{Var}}
\newcommand{\Cov}{\operatorname{Cov}}
\newcommand{\Z}{\mathbb Z}
\newcommand{\N}{\mathbb N}
\newcommand{\R}{\mathbb R}

\newcommand{\one}{\mathbf 1}
\newcommand{\dW}{d_{\mathrm W}}
\newcommand{\dK}{d_{\mathrm K}}
\newcommand{\dTV}{d_{\mathrm{TV}}}
\newcommand{\Poi}{\operatorname{Poi}}
\newcommand{\U}{\mathrm U}
\newcommand{\D}{\mathrm D}

\title[Gaps between roots in ballistic deposition]{Gaps between roots in ballistic deposition on a strip:\\
linear variance, Gaussian limits and a Poisson regime}

\author{Mateo Matijasevick}
\address{Purdue University}
\email{mmatijas@purdue.edu}

\date{October 20, 2026}

\begin{document}

\begin{abstract}
Mansour, Rastegar and Roitershtein studied the \emph{roots} of ballistic deposition on a
cyclic strip of $K$ sites, namely the sites that receive a particle at height one, and the
number $D_{i,K}$ of pairs of consecutive roots at distance $i+1$. They computed the first two
moments of $D_{i,K}$ for $i\le7$. They conjectured that for every $i$ the mean and the variance
are exactly linear in $K$, that $D_{i,K}$ satisfies a central limit theorem, and that the
standardized vector $(D_{i,K})_{i\ge1}$ converges to a Gaussian process. We prove these
conjectures. The key observation is that the roots are exactly the valleys of a uniformly
random circular arrangement, so that $D_{i,K}$ counts occurrences of a family of consecutive
patterns. We determine sizes from which the moments are exactly linear; the size for the mean is sharp
for every $i$, and those for the second moments cannot be lowered in any case we checked. We give closed forms
for the variance constants and for the limiting covariance; the mean constants agree with a known
distribution of distances between surface peaks. The reduction and the exact moment formulas are
formalized in Lean~4 with Mathlib. We also
prove a Wasserstein bound of order $(\E D_{i,K})^{-1/2}$ that is uniform in $i$, and a local
limit theorem. Finally we describe the limit completely when $i$ grows with $K$: it is
degenerate, Poisson or Gaussian according as $\E D_{i,K}$ tends to zero, to a positive
constant, or to infinity.
\end{abstract}

\maketitle

%=========================================================================
\section{Introduction}\label{sec:intro}

\subsection{The model and the conjectures}
Let $K\ge3$ and consider the cycle $\Z_K=\{0,1,\dots,K-1\}$, in which $x$ has neighbours
$x\pm1$ modulo $K$. The \emph{ballistic deposition process} on $\Z_K$~\cite{MRR19} starts from
the flat profile $H_0\equiv0$. At step $n\ge1$ a site $X_n$ is chosen uniformly at random,
independently of the past, and only its height changes:
\begin{equation}\label{eq:bd}
H_n(X_n)=1+\max\{H_{n-1}(X_n-1),\,H_{n-1}(X_n),\,H_{n-1}(X_n+1)\}.
\end{equation}
A site is a \emph{root} if at some step it receives a particle at height one. Let
$r_1<\dots<r_R$ be the roots, $R=R^{[K]}$, and for $i\ge1$ let
\[
D_{i,K}=\sum_{j=1}^{R-1}\one\{r_{j+1}-r_j=i+1\}+\one\{r_R-r_1=K-i-1\}
\]
be the number of pairs of cyclically consecutive roots at distance $i+1$, the wrap-around pair
included. Mansour, Rastegar and Roitershtein~\cite{MRR19} proved that $\E R^{[K]}=K/3$,
$\Var R^{[K]}=2K/45$ for $K\ge7$, and that $R^{[K]}$ satisfies a central limit theorem. They
also found $\E D_{1,K}=2K/15$ and $\Var D_{1,K}=1772K/14175$ for $K\ge9$, together with the
corresponding linear formulas for $2\le i\le7$ when $K\ge31$. Their method uses
generating functions and Riccati equations. On this evidence they conjectured the following
\cite[Section~3.5]{MRR19}. (In continuous time, with independent rate-one clocks at the sites, the
same growth rule on a cylinder was introduced by Atar, Athreya and Kang~\cite{AAK01}; see
Section~\ref{sec:further} for the growth speed.)

\begin{conjecture}[\cite{MRR19}]\label{conj}
\begin{enumerate}[label=(\roman*)]
\item For every $i\ge1$ there are $K_i\in\N$ and positive rationals $a_i,b_i$ such that
$\E D_{i,K}=a_iK$ and $\sigma(D_{i,K})=b_iK$ for all $K>K_i$.
\item The standardized variables $(D_{i,K}-\E D_{i,K})/\sigma(D_{i,K})$ converge in
distribution to $N(0,1)$.
\item The vector of standardized variables converges weakly, in the product space $\R^\N$, to a
Gaussian process.
\end{enumerate}
\end{conjecture}

Here $\sigma$ denotes the standard deviation. A sum of $K$ bounded, locally dependent terms has
standard deviation $O(\sqrt K)$, so in part~(i) the standard deviation must be replaced by the
variance $\sigma^2$, as in Theorems~3.3 and~3.5 of~\cite{MRR19}. We prove this version.

\subsection{Main results}
Throughout, $g=i+1$ denotes the gap length and $Z$ a standard normal variable. We write
$W_{i,K}=(D_{i,K}-a_iK)/\sqrt{b_iK}$. For random variables $X,Y$, $\dW(X,Y)$ denotes the
Wasserstein distance, $\dK$ the Kolmogorov distance and $\dTV$ the total variation distance.

\begin{mainthm}[exact linear moments]\label{thm:moments}
For every $i\ge1$ there are positive rationals $a_i$ and $b_i$ such that
\[
\E D_{i,K}=a_iK\quad(K\ge i+4),\qquad \Var D_{i,K}=b_iK\quad(K\ge 2i+7).
\]
More generally, $\Cov(D_{i,K},D_{j,K})=c_{ij}K$ for $K\ge i+j+7$, with $c_{ii}=b_i$. The
constants are given by
\[
a_i=\frac{i\,2^{i+1}}{(i+1)(i+2)(i+4)\,i!},\qquad b_i=a_i-(2i+7)a_i^2+2u_i+2v_i ,
\]
with $u_i,v_i$ the explicit expressions~\eqref{eq:u}--\eqref{eq:v}, and $c_{ij}$ is given
in~\eqref{eq:cij}.
\end{mainthm}

\begin{mainthm}[Gaussian limits]\label{thm:gauss}
\begin{enumerate}[label=(\alph*)]
\item For every $i\ge1$, $W_{i,K}\Rightarrow N(0,1)$ as $K\to\infty$.
\item Let $\mathbf V_K\in\R^\N$ have coordinates $(D_{i,K}-\E D_{i,K})/\sqrt{\Var D_{i,K}}$, set to $0$ when
the variance vanishes or $i\ge K$. Then $\mathbf V_K$ converges weakly in $\R^\N$ to the centered
Gaussian law with covariance $\Sigma_{ij}=c_{ij}/\sqrt{b_ib_j}$.
\end{enumerate}
\end{mainthm}

\begin{mainthm}[rates]\label{thm:rates}
For all $i\ge1$ and $K\ge2i+7$,
\[
\dW(W_{i,K},Z)\le\frac{133}{\sqrt{a_iK}} .
\]
For fixed $i$ a local limit theorem holds:
$\sup_k|\sqrt{b_iK}\,\Pp(D_{i,K}=k)-\varphi((k-a_iK)/\sqrt{b_iK})|=O_i(\log K/\sqrt K)$, where
$\varphi$ is the standard normal density.
\end{mainthm}

\begin{mainthm}[growing gap length]\label{thm:poisson-intro}
For all $i\ge1$ and $K\ge2i+7$,
$\dTV(D_{i,K},\Poi(a_iK))\le\min(1,a_iK)\,\varepsilon_i$, where $\varepsilon_i$ is explicit
(Theorem~\ref{thm:poisson}) and $\varepsilon_i\sim\frac52ia_i\to0$.
Consequently, if $1\le i(K)\le K-4$ and $\lambda_K=a_{i(K)}K$, then $D_{i(K),K}\to0$ in
probability when $\lambda_K\to0$. It converges in total variation to $\Poi(\lambda)$ when
$\lambda_K\to\lambda\in(0,\infty)$, and $W_{i(K),K}\Rightarrow N(0,1)$ when
$\lambda_K\to\infty$.
\end{mainthm}

The size thresholds in Theorem~\ref{thm:moments} cannot be lowered in any case we checked
(Remark~\ref{rem:sharp}). The second threshold agrees with $K\ge9$ for $i=1$ in~\cite{MRR19}.

\subsection{The idea, and relation to the literature}
The proofs rest on a simple observation (Lemma~\ref{lem:root}). A site becomes a root exactly
when it is visited before both of its neighbours. If each site carries an independent rate-one
Poisson clock, the order of first visits is a uniformly random circular arrangement, and the
roots are its \emph{valleys}, the strict local minima. After this reduction the growth process
plays no further role. The statistic $R^{[K]}$ becomes the number of valleys of a random
circular permutation, that is, the number of local minima of an i.i.d.\ field on the cycle graph.
So the results of~\cite{MRR19} on $R^{[K]}$ are special cases of classical results on local
extremes of random functions on graphs~\cite{BRS89,BR89} and on peaks of random permutations
(see, e.g.,~\cite{WS96}).

The same growth rule on a ring was studied by Hivert, Nechaev, Oshanin and
Vasilyev~\cite{HNOV07}. They mapped its steady-state surface to rise-and-descent patterns of random
permutations, computed the law of the number of surface peaks, and found that two consecutive
surface peaks lie at distance $l$ with probability $p(l)=2^l(l-1)(l+2)/(l+3)!$ per site; according
to~\cite{HNOV07}, the same law describes the distances between weak bonds in a random Ising
chain~\cite{DG86}. A time-reversal duality (Remark~\ref{rem:dual}) explains the connection: for
every $n$, the set of surface peaks after $n$ steps has the same law as the set of roots formed by
step~$n$. Hence $p(i+1)=a_i$ below, and on a ring the number of surface peaks has the same law as the
number of roots $R^{[K]}$ of~\cite{MRR19}. The cluster forest built by this growth rule is
studied in~\cite{KNOSV10}. Similarly,
$D_{i,K}$ counts occurrences of the consecutive patterns of length $i+4$ whose up/down word is
$\D\,\U^a\D^b\,\U$ with $a,b\ge1$ and $a+b=i+1$. Here the up/down word of a sequence of distinct
numbers records $\U$ for each ascent and $\D$ for each descent between consecutive entries. Thus a
gap of length $i+1$ is an ascending run of length $a$ followed by a descending run of length $b$.
The joint behaviour of runs up and down is a classical subject~\cite{LW44}. Central limit theorems for
such local statistics follow from the theory of $m$-dependent sequences~\cite{HR48}; see
\cite{EN03} for the enumeration of consecutive patterns and~\cite{Gol05} for Berry--Esseen bounds for
pattern counts. The formula for $a_i$ is the distance law $p(i+1)$ of~\cite{HNOV07}, and it is
equivalent to the enumeration of permutations with peak set $\{2,n-1\}$ due to Billey, Burdzy and
Sagan~\cite{BBS13} (Remark~\ref{rem:bbs}).

We have not found the reduction for roots (first visits) in~\cite{MRR19,MRR20}, nor in the works
citing~\cite{MRR19}; its time-reversed counterpart for surface peaks is implicit in~\cite{HNOV07}.
We could not find the second-order results below (the constants $b_i$ and $\Sigma$, the exact
covariance structure) in~\cite{HNOV07}, which treats only the marginal law $p(l)$, nor in the
literature on runs and patterns.
What the reduction does not give automatically is the following, and it constitutes the
substance of this paper:
\begin{itemize}[leftmargin=*]
\item the non-degeneracy $b_i>0$ for every $i$;
\item the exact covariance structure, with size thresholds that cannot be lowered in any case we
checked;
\item closed forms for $b_i$ and $\Sigma$;
\item normal approximation bounds that are uniform in $i$;
\item the Poisson regime.
\end{itemize}

\subsection{Organization}
Section~\ref{sec:valleys} proves the reduction. Section~\ref{sec:moments} computes the
moments exactly, and Section~\ref{sec:constants} evaluates the constants.
Section~\ref{sec:normal} proves the normal approximation results, and
Section~\ref{sec:poisson} treats growing gap lengths. Section~\ref{sec:further} collects
further results on the root count, the growth speed and nonuniform rates.
Appendix~\ref{app:computations} describes the computer verification.

%=========================================================================
\section{Roots as valleys}\label{sec:valleys}

\begin{lemma}[root criterion]\label{lem:root}
Let $K\ge3$ and let $x_1,\dots,x_n$ be any finite sequence of sites, deposited in this order
on the flat profile according to~\eqref{eq:bd}. For a site $y$ let $\iota(y)$ be the index of
its first occurrence, with $\iota(y)=\infty$ if $y$ does not occur. Then $x$ receives a particle at
height one if and only if $\iota(x)<\iota(x-1)$ and $\iota(x)<\iota(x+1)$.
\end{lemma}

\begin{proof}
Heights never decrease, and a deposit raises the height of its site to at least one. Hence at
every step $H(y)=0$ exactly for the sites $y$ not yet visited. A particle deposited at $x$
lands at height one if and only if $H(x-1)=H(x)=H(x+1)=0$, that is, if and only if none of
$x-1,x,x+1$ has been visited before. This can only happen at the first visit of $x$, and it
happens then if and only if neither neighbour was visited earlier.
\end{proof}

\begin{lemma}[Poissonization]\label{lem:poisson}
Let $(N_x)_{x\in\Z_K}$ be independent rate-one Poisson processes, let $T_x$ be the first
arrival time of $N_x$, and put $U_x=1-e^{-T_x}$. Then the sequence of labels of the jumps of
$\sum_xN_x$ is i.i.d.\ uniform on $\Z_K$, and therefore drives a copy of the deposition
process. In this copy, almost surely, the set of roots equals
\[
\{x\in\Z_K:\ U_x<U_{x-1}\text{ and }U_x<U_{x+1}\},
\]
where $(U_x)_{x\in\Z_K}$ are i.i.d.\ uniform on $(0,1)$.
\end{lemma}

\begin{proof}
The label sequence of a superposition of independent Poisson processes is i.i.d., with
probabilities proportional to the rates. Almost surely all sites are visited and the $T_x$ are
distinct. After time $\max_xT_x$ all heights are positive, so no later particle lands at height
one. Hence the roots are determined by a finite prefix of the label sequence, to which
Lemma~\ref{lem:root} applies. First occurrences in the label sequence are ordered like the
$T_x$, and $t\mapsto1-e^{-t}$ is increasing.
\end{proof}

Thus the root set has the law of the set of valleys of $K$ i.i.d.\ uniform variables placed on
a cycle, equivalently of a uniformly random circular arrangement of $K$ distinct numbers. From
now on we work with this representation; see Figure~\ref{fig:reduction}.

\begin{figure}[t]
\centering
\includegraphics[width=0.86\textwidth]{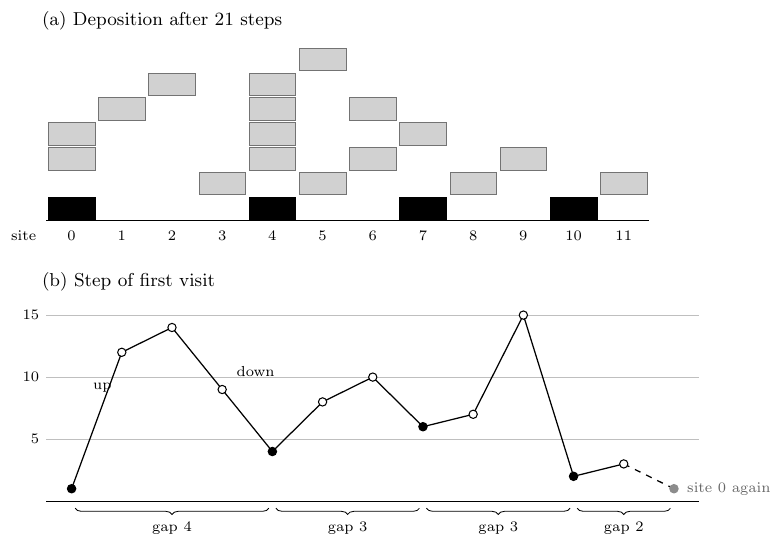}
\caption{Roots are valleys. (a)~Ballistic deposition on $K=12$ sites after $21$ steps; the strip is
cyclic, so site $11$ neighbours site $0$. The particles that landed at height one, which mark the
roots, are black. (b)~The step at which each site was first visited. The filled dots are the
strict local minima (valleys); they are exactly the roots (Lemma~\ref{lem:root}). Between two
consecutive valleys the first-visit times go up and then down (Lemma~\ref{lem:unimodal}). The
braces give the gap lengths, so here $D_{1,12}=1$, $D_{2,12}=2$ and $D_{3,12}=1$.}
\label{fig:reduction}
\end{figure}
 Let
$\xi_x=\one\{U_x<U_{x-1},\,U_x<U_{x+1}\}$ and, for $g=i+1$,
\begin{equation}\label{eq:Y}
Y_j=Y_j^{(g)}=\xi_j\,\xi_{j+g}\prod_{t=1}^{g-1}(1-\xi_{j+t})\qquad(j\in\Z_K).
\end{equation}
If $K\ge g+1$, then $Y_r=1$ exactly when $r$ is a root and the next root in the positive
direction is $r+g$. Hence
\begin{equation}\label{eq:DY}
D_{i,K}=\sum_{j\in\Z_K}Y_j^{(i+1)}\qquad(K\ge i+2).
\end{equation}
When there is a single root, both sides vanish. The indicator $Y_j$ is a function of the
priorities in the window $\{j-1,\dots,j+g+1\}$ of $m:=g+3=i+4$ sites. The same definitions
make sense on $\Z$, with $(U_x)_{x\in\Z}$ i.i.d.\ uniform. There $(Y_j)_{j\in\Z}$ is stationary,
and $Y_j$ is independent of $(Y_k)_{|k-j|\ge m}$.

\begin{lemma}[local structure]\label{lem:local}
On $\Z$, let $g,g'\ge2$ and $h\in\Z$.
\begin{enumerate}[label=(\alph*)]
\item No two adjacent sites are both valleys.
\item $Y^{(g)}_0Y^{(g')}_h\equiv0$ unless $h\in\{g,\,g+2,\,-g',\,-g'-2\}$, or $h=0$ and $g=g'$, or
$h\notin[-g'-2,\,g+2]$.
\end{enumerate}
\end{lemma}

\begin{proof}
(a) If $x$ is a valley then $U_{x+1}>U_x$, so $x+1$ is not a valley.

(b) $Y^{(g)}_0=1$ means that $0$ and $g$ are valleys and no site strictly between them is one.
Suppose first that $0<h<g$; then $Y^{(g')}_h=1$ requires a valley at $h$, a contradiction. If
$h=g+1$, the sites $g$ and $g+1$ would be adjacent valleys, contradicting~(a). The cases
$-g'<h<0$ and $h=-g'-1$ are symmetric. Finally, if $h=0$ and $g\ne g'$, the next valley after
$0$ would be both $g$ and $g'$.
\end{proof}

\begin{remark}[a finite form of Lemma~\ref{lem:poisson}]\label{rem:finite}
Poisson clocks are not needed. For every $n$, consider the words $x_1,\dots,x_n$ in which every
site occurs. Relabelling the sites by a permutation $\rho$ maps this set of words onto itself and
composes the order of first visits with $\rho$. Hence every circular arrangement is the order of
first visits of exactly $1/K!$ of these words, and by Lemma~\ref{lem:root}, conditionally on all
sites being visited by step $n$, the root set has exactly the law of the valley set of a uniformly
random arrangement. The root set no longer changes after every site has been visited. Both
statements are formalized in Lean (Appendix~\ref{app:computations}).
\end{remark}

\begin{remark}[time-reversal duality]\label{rem:dual}
A site $x$ is a strict local maximum of $H_n$ if and only if its last visit among $x_1,\dots,x_n$ is
later than the last visits of $x-1$ and $x+1$ (an unvisited site counts as visited at time
$-\infty$). Indeed, if the last visit of $x$ comes after those of its neighbours, the deposit at
that time puts $x$ above both neighbours and nothing changes afterwards; if, say, $x+1$ is visited
last, its final deposit lands above $x$. Reversing the word turns last visits into first visits,
so by Lemma~\ref{lem:root} the peaks of $H_n$ are the roots of the reversed word. Since an i.i.d.\
word and its reversal have the same law, for every $n$ the peak set of $H_n$ has the same law as
the set of roots formed by step $n$. This identifies the root statistics of~\cite{MRR19} with the
surface-peak statistics of~\cite{HNOV07}; it is also formalized in Lean.
\end{remark}

%=========================================================================
\section{Exact moments}\label{sec:moments}

On $\Z$, define for $i,j\ge1$, with $g=i+1$ and $g'=j+1$:
\begin{align*}
a_i&=\E Y_0^{(g)},\qquad u_{ij}=\E\bigl[Y^{(g)}_0Y^{(g')}_g\bigr],\qquad
v_{ij}=\E\bigl[Y^{(g)}_0Y^{(g')}_{g+2}\bigr],
\end{align*}
and $u_i=u_{ii}$, $v_i=v_{ii}$. Thus $u_{ij}$ is the probability that consecutive gaps of lengths
$g$ and $g'$ start at $0$, and $v_{ij}$ the probability of consecutive gaps $g,2,g'$. By reflection
$u_{ij}=u_{ji}$ and $v_{ij}=v_{ji}$. Put
\begin{equation}\label{eq:cij}
c_{ij}=\sum_{h\in\Z}\Cov\bigl(Y^{(i+1)}_0,Y^{(j+1)}_h\bigr)
=\one\{i=j\}\,a_i+2u_{ij}+2v_{ij}-(i+j+7)\,a_ia_j ,
\end{equation}
and $b_i=c_{ii}=a_i-(2i+7)a_i^2+2u_i+2v_i$. The second equality in~\eqref{eq:cij} holds
because the windows of $Y_0^{(g)}$ and $Y_h^{(g')}$ intersect for exactly the $g+g'+5=i+j+7$
offsets $h\in[-g'-2,g+2]$. Lemma~\ref{lem:local} identifies the nonzero products, and the
covariance vanishes for all other $h$.

\begin{proposition}[exact moments]\label{prop:moments}
For all $i,j\ge1$:
\begin{enumerate}[label=(\alph*)]
\item $\E D_{i,K}=a_iK$ for $K\ge i+4$;
\item $\Cov(D_{i,K},D_{j,K})=c_{ij}K$ for $K\ge i+j+7$; in particular $\Var D_{i,K}=b_iK$ for $K\ge2i+7$.
\end{enumerate}
\end{proposition}

\begin{proof}
(a) For $K\ge m=i+4$ the window of each $Y_j$ consists of $m$ distinct sites of the cycle, so
$\E Y_j=a_i$, and~\eqref{eq:DY} gives the claim.

(b) Write
$\Cov(D_{i,K},D_{j,K})=\sum_{x\in\Z_K}\sum_{h\in\Z_K}\Cov(Y^{(g)}_x,Y^{(g')}_{x+h})$. If
$K\ge i+j+7$, the $i+j+7$ offsets $h\in[-g'-2,g+2]$ are distinct modulo $K$. For each of them
the union of the two windows consists of at most $i+j+7\le K$ consecutive sites. So the
covariance equals the corresponding covariance on $\Z$. For every other residue the two windows
are disjoint on the cycle, and the covariance vanishes. Summing over $x$ gives $c_{ij}K$.
\end{proof}

\begin{remark}[sharpness]\label{rem:sharp}
The threshold for the mean is sharp for every $i$. For $K=i+3$ the two outer sites of the window
of $Y_j$ coincide, and this common neighbour of the two valleys must exceed both of them. In the
notation of Lemma~\ref{lem:kernel} one finds
$\E D_{i,i+3}=(i+3)\iint\min(s,t)\,k_i(s,t)\,ds\,dt$, while $a_i=\iint st\,k_i(s,t)\,ds\,dt$
(Theorem~\ref{thm:a}). Since $\min(s,t)>st$ and $k_i(s,t)>0$ on $(0,1)^2$, we get
$\E D_{i,i+3}>(i+3)a_i$; for instance $\E D_{1,4}=\frac23>\frac8{15}$, as in~\cite[Thm.~3.3]{MRR19}.
For the second moments, exhaustive enumeration of all $(K-1)!$ circular arrangements, for
$K\le14$, shows that the thresholds cannot be lowered in any case covered: the variance formula
fails at $K=2i+6$ for $i=1,\dots,4$, and the covariance formula fails at $K=i+j+6$ for all pairs
$i<j$ with $i+j+6\le14$. See Appendix~\ref{app:computations}.
\end{remark}

%=========================================================================
\section{The constants}\label{sec:constants}

\begin{lemma}\label{lem:unimodal}
Let $w_0,\dots,w_g$ be distinct with $w_0<w_1$ and $w_{g-1}>w_g$. Then none of
$w_1,\dots,w_{g-1}$ is a strict local minimum if and only if the sequence increases up to its
maximum and then decreases.
\end{lemma}

\begin{proof}
If a descent $w_k>w_{k+1}$ precedes an ascent $w_l<w_{l+1}$, $k<l$, then the minimum of $w$ on
$\{k,\dots,l+1\}$ is attained in $\{k+1,\dots,l\}$, where it is a strict local minimum. The
converse is clear.
\end{proof}

\begin{lemma}[excursion kernel]\label{lem:kernel}
Let $\eta_1,\dots,\eta_n$ be i.i.d.\ uniform on $(0,1)$, $n\ge1$, and $x,y\in[0,1]$. Then
\[
\Pp\bigl(\eta_1>x,\ \eta_n>y,\ (\eta_1,\dots,\eta_n)\text{ unimodal}\bigr)=\frac{(2-x-y)^n-|x-y|^n}{2\,n!}.
\]
In the variables $s=1-x$ and $t=1-y$ this equals
$k_n(s,t):=\frac{(s+t)^n-|s-t|^n}{2\,n!}$, and
$\sum_{n\ge1}k_n(s,t)z^n=e^{z\max(s,t)}\sinh\bigl(z\min(s,t)\bigr)$.
\end{lemma}

\begin{proof}
Condition on the set of values; all $n!$ orders are then equally likely. In a unimodal order
satisfying the constraints, the maximum $M$ sits at the peak. Every other value $v$ is placed
either to its left, which is admissible iff $v>x$, or to its right, admissible iff $v>y$, and
these choices determine the order. Moreover $M>\max(x,y)$. Hence the number of admissible
orders is $\one\{M>x\vee y\}\prod_{v\ne M}(\one\{v>x\}+\one\{v>y\})$. The maximum has density
$nM^{n-1}$, and given $M$ the other values are i.i.d.\ uniform on $(0,M)$. Therefore the
probability equals
$\frac1{n!}\int_{x\vee y}^1n(2M-x-y)^{n-1}dM=\frac{(2-x-y)^n-|x-y|^n}{2\,n!}$.
\end{proof}

Write $\alpha_i(t)=\int_0^1s\bigl((s+t)^i-|s-t|^i\bigr)\,ds$ and $\beta_i(x)=\int_x^1\alpha_i(t)\,dt$.
A direct computation, in which the terms in $t^{i+2}$ cancel, gives
\begin{equation}\label{eq:alpha}
\alpha_i(t)=\frac{(1+t)^{i+2}-(1-t)^{i+2}}{i+2}-t\,\frac{(1+t)^{i+1}+(1-t)^{i+1}}{i+1}.
\end{equation}

\begin{theorem}[the mean constant]\label{thm:a}
For every $i\ge1$,
\[
a_i=\frac{1}{2\,i!}\int_0^1t\,\alpha_i(t)\,dt=\frac{i\,2^{i+1}}{(i+1)(i+2)(i+4)\,i!},
\]
and
\[
\sum_{i\ge1}a_iz^i=\frac{2z^3+3z^2(e^{2z}+1)-6ze^{2z}+3e^{2z}-3}{6z^4}.
\]
In particular $a_i\sim2^{i+1}/(i^2\,i!)$, $\sum_ia_i=\frac13$ and $\sum_i(i+1)a_i=1$.
\end{theorem}

\begin{proof}
$Y_0=1$ means that $0$ and $g$ are valleys, i.e.\ $U_{-1}>U_0<U_1$ and $U_{g-1}>U_g<U_{g+1}$, and
that no site in $\{1,\dots,g-1\}$ is a valley. By Lemma~\ref{lem:unimodal}, given $U_0=1-s$ and
$U_g=1-t$, the conditions on $U_1,\dots,U_{g-1}$ say exactly that $(U_1,\dots,U_{g-1})$ is unimodal
with $U_1>U_0$ and $U_{g-1}>U_g$. By Lemma~\ref{lem:kernel},
with $n=g-1=i$, and since $\Pp(U_{-1}>U_0\mid U_0)=s$ and $\Pp(U_{g+1}>U_g\mid U_g)=t$,
\[
a_i=\int_0^1\!\!\int_0^1st\,k_i(s,t)\,ds\,dt=\frac1{2\,i!}\int_0^1t\,\alpha_i(t)\,dt .
\]
Inserting~\eqref{eq:alpha}, the integral reduces to the elementary integrals
$\int_0^1t^r(1\pm t)^k\,dt$ ($r=1,2$) and gives $\int_0^1t\,\alpha_i=i\,2^{i+2}/((i+1)(i+2)(i+4))$.
The generating function follows by summing the series, or from
$\sum_ia_iz^i=\iint st\,e^{z\max(s,t)}\sinh(z\min(s,t))\,ds\,dt$. The two sums express that
roots have density $\frac13$ and that the gaps tile the cycle.
\end{proof}

\begin{remark}[prior art]\label{rem:bbs}
Reversing the order of the values maps the up/down word $\D\,\U^a\D^b\,\U$ to
$\U\,\D^a\U^b\,\D$. Hence $n!\,a_{n-4}$ is the number of permutations of $[n]$ with peak set
exactly $\{2,n-1\}$. Billey, Burdzy and Sagan~\cite[Thm.~4.2]{BBS13} showed that this number equals
$(n-1)(n-4)2^{n-3}$. The same numbers, $a_i=p(i+1)$ with $p(l)=2^l(l-1)(l+2)/(l+3)!$, together with
their generating function, appear in~\cite[Eq.~(69)]{HNOV07} as the probability that two
consecutive surface peaks of the ring are at distance $l$ (compare Remark~\ref{rem:dual});
\cite{HNOV07} attributes the formula also to~\cite{DG86}. So Theorem~\ref{thm:a} is not new; we
keep the kernel proof because the kernel is needed for the second moments.
\end{remark}

\begin{theorem}[the second-moment constants]\label{thm:b}
For $i,j\ge1$,
\[
u_{ij}=\frac{1}{4\,i!\,j!}\int_0^1\alpha_i(t)\alpha_j(t)\,dt,\qquad
v_{ij}=\frac{1}{4\,i!\,j!}\int_0^1\beta_i(x)\beta_j(x)\,dx .
\]
In particular
\begin{align}
u_i&=\frac{2\cdot4^i(2i^3+5i^2+4i+2)}{(i!)^2(i+1)^2(i+2)^2(2i+3)(2i+5)}
-\frac{4^i}{(2i)!\,(i+1)(2i+3)(2i+5)},\label{eq:u}\\
v_i&=\frac{4\cdot4^i(2i-1)(2i^2+7i+4)}{(i!)^2(i+1)^2(i+2)^2(i+4)(2i+5)(2i+7)}
+\frac{2\cdot4^i}{(2i+1)!\,(i+1)(i+2)(2i+5)(2i+7)}.\label{eq:v}
\end{align}
As $i\to\infty$, $u_i\sim\frac i4a_i^2$, $v_i\sim a_i^2$ and
$b_i=a_i-\bigl(\frac32+o(1)\bigr)i\,a_i^2$.
\end{theorem}

\begin{proof}
Condition on the value $U_g=1-t$ at the common valley of the two consecutive gaps. Given it,
the two excursions use disjoint sets of variables and are conditionally independent. Each has
the kernel of Lemma~\ref{lem:kernel}, and the valley condition at $g$ is contained in the two
kernels. Integrating out the outer valleys gives
$\int_0^1s\,k_i(s,t)\,ds=\alpha_i(t)/(2\,i!)$, whence the formula for $u_{ij}$. For $v_{ij}$ the
middle gap of length $2$ contributes the kernel $k_1(t,w)=\min(t,w)=\int_0^1\one\{x<t\}\one\{x<w\}\,dx$.
Fubini's theorem then gives
$\iint\alpha_i(t)\min(t,w)\alpha_j(w)\,dt\,dw=\int_0^1\beta_i\beta_j$.

For $i=j$, insert~\eqref{eq:alpha}, and the corresponding expression for $\beta_i$, and expand the squares.
Every term is a constant times $t^r(1+t)^k$, $t^r(1-t)^k$ or $t^r(1-t^2)^k(1\pm t)^e$ with
$r\le2$, $e\in\{0,1\}$ and $k$ affine in $i$. These integrate as follows:
$\int_0^1t^r(1+t)^k\,dt$ by writing $t^r=((1+t)-1)^r$;
$\int_0^1t^r(1-t)^k\,dt=r!\,k!/(k+r+1)!$; $\int_0^1t^{2\rho+1}(1-t^2)^k\,dt=\rho!\,k!/(2(\rho+k+1)!)$;
and $\int_0^1t^{2\rho}(1-t^2)^k\,dt=W_k\prod_{s=1}^{\rho}\frac{2s-1}{2k+2s+1}$, where
$W_k=4^k(k!)^2/(2k+1)!$. Collecting terms, a mechanical computation that we carried out by
computer algebra symbolically in $i$, gives~\eqref{eq:u} and~\eqref{eq:v}. These were also checked
against the exact permutation counts for $i\le30$ (Appendix~\ref{app:computations}). The
asymptotics follow from these formulas and Theorem~\ref{thm:a}.
\end{proof}

\begin{corollary}[positivity and comparison]\label{cor:pos}
For every $i\ge1$:
\begin{enumerate}[label=(\alph*)]
\item $a_i$ and $b_i$ are rational, and $b_i\ge a_i\bigl(1-(2i+7)a_i\bigr)$;
\item $(2i+7)a_i\le\frac{11}9$ and $(6i+19)a_i\le\frac{31}9$, with equality at $i=2$;
\item $b_i\ge\frac23a_i>0$, and $b_i/a_i\to1$.
\end{enumerate}
\end{corollary}

\begin{proof}
Each of $a_i,u_i,v_i$ is the probability of an event determined by the relative order of
finitely many i.i.d.\ continuous variables, hence rational. The inequality in~(a) holds because
$u_i,v_i\ge0$.

For~(b), observe that
$a_{i+1}/a_i=\frac{2(i+1)(i+4)}{i(i+3)(i+5)}$, which is decreasing in $i$ and equals $\frac{18}{35}$
at $i=2$. Hence $\frac{(2i+9)a_{i+1}}{(2i+7)a_i}\le\frac{13}{11}\cdot\frac{18}{35}<1$ and
$\frac{(6i+25)a_{i+1}}{(6i+19)a_i}\le\frac{37}{31}\cdot\frac{18}{35}<1$ for $i\ge2$. So both
sequences decrease from $i=2$ on. At $i=1,2$ their values are $\frac65,\frac{11}9$ and
$\frac{10}3,\frac{31}9$ respectively.

For~(c), when $i\ge4$ we have $(2i+7)a_i\le15a_4=\frac13$, and~(a) gives $b_i\ge\frac23a_i$.
For $i=1,2,3$ the exact values in Table~\ref{tab:constants} give $b_i/a_i\ge0.71$. The limit
follows from Theorem~\ref{thm:b}.
\end{proof}

\begin{table}[t]
\centering\small
\caption{Exact constants. The first seven pairs agree with~\cite{MRR19}.}\label{tab:constants}
\begin{tabular}{rlll}
\toprule
$i$ & $a_i$ & $b_i$ & $b_i/a_i$\\\midrule
1 & $2/15$ & $1772/14175$ & $0.93757$\\
2 & $1/9$ & $32/405$ & $0.71111$\\
3 & $2/35$ & $119732/2837835$ & $0.73835$\\
4 & $1/45$ & $12154/637875$ & $0.85743$\\
5 & $4/567$ & $649555688/97692469875$ & $0.94249$\\
6 & $1/525$ & $5967328/3192564375$ & $0.98129$\\
7 & $2/4455$ & $191501338988/428772250281375$ & $0.99486$\\
8 & $4/42525$ & $262821932/2797557091875$ & $0.99877$\\
9 & $4/225225$ & $284536854284888/16025362854266390625$ & $0.99974$\\
10 & $2/654885$ & $6186884112998/2025948882579496875$ & $0.99995$\\
\bottomrule
\end{tabular}
\end{table}

\begin{remark}
For example $c_{12}=-\frac{22}{945}$, $c_{13}=-\frac{3182}{155925}$ and $c_{23}=-\frac{28}{1215}$.
The correlations $\Sigma_{ij}$, $1\le i<j\le4$, all lie in $[-0.40,-0.22]$. Gap counts of
different lengths are therefore negatively correlated. This reflects the deterministic
identity $\sum_{i\ge1}(i+1)D_{i,K}=K$. It gives $\sum_i(i+1)\Cov(D_{i,K},D_{j,K})=0$ for every $K$,
and letting $K\to\infty$ yields $\sum_{i\ge1}(i+1)c_{ij}=0$ for every $j$; the terms with
$K<i+j+7$ are negligible because $a_i$ decays superexponentially.
\end{remark}

%=========================================================================
\section{Normal approximation}\label{sec:normal}

We use the following form of Stein's method for dependency neighbourhoods (compare
\cite[Thm.~3.5]{Ros11}).

\begin{lemma}\label{lem:stein}
Let $X_1,\dots,X_n$ be centered random variables with $\sigma^2=\Var\sum_jX_j>0$. Suppose
that for each $j$ there is a set $\mathcal N_j\ni j$ such that $X_j$ is independent of
$(X_k)_{k\notin\mathcal N_j}$, and put $V_j=\sum_{k\in\mathcal N_j}X_k$. Then
\[
\dW\Bigl(\frac1\sigma\sum_jX_j,\,Z\Bigr)\le\frac1{\sigma^3}\sum_j\E\bigl[|X_j|V_j^2\bigr]
+\frac{\sqrt{2/\pi}}{\sigma^2}\sqrt{\Var\Bigl(\sum_jX_jV_j\Bigr)}.
\]
\end{lemma}

\begin{proof}
Let $h$ be $1$-Lipschitz and let $f$ solve $f'(w)-wf(w)=h(w)-\E h(Z)$. Then
$\|f'\|_\infty\le\sqrt{2/\pi}$ and $f'$ is $2$-Lipschitz~\cite[Lemma~2.5]{Ros11}. Write
$W=\sigma^{-1}\sum X_j$ and $W_j=W-V_j/\sigma$, which is independent of $X_j$. Then
$\E Wf(W)=\sigma^{-1}\sum_j\E X_j(f(W)-f(W_j))$. Taylor's formula gives
$\E Wf(W)=\E[f'(W)\Theta]+\mathcal R$, with $\Theta=\sigma^{-2}\sum_jX_jV_j$ and
$|\mathcal R|\le\sigma^{-3}\sum_j\E|X_j|V_j^2$. Since $\E X_jX_k=0$ for $k\notin\mathcal N_j$, we have
$\E\Theta=1$. Hence
$|\E h(W)-\E h(Z)|=|\E f'(W)(1-\Theta)-\mathcal R|\le\sqrt{2/\pi}\,(\Var\Theta)^{1/2}+|\mathcal R|$.
\end{proof}

\begin{theorem}\label{thm:wass}
For all $i\ge1$ and $K\ge2i+7$, $\dW(W_{i,K},Z)\le133/\sqrt{a_iK}$. Consequently
$\dK(W_{i,K},Z)\le10.4\,(a_iK)^{-1/4}$.
\end{theorem}

\begin{proof}
Let $m=i+4$ and $g=i+1$, and apply Lemma~\ref{lem:stein} on $\Z_K$ with $X_j=Y_j-a_i$ and
$\mathcal N_j=\{k:d(j,k)\le m-1\}$, where $d$ is the cyclic distance. For $K\ge2m-1$ these are sets of
$2m-1$ sites, and $X_j$ is independent of $(X_k)_{k\notin\mathcal N_j}$ because the corresponding
windows are disjoint. By Proposition~\ref{prop:moments}, $\sigma^2=b_iK$.

The starts of distinct gaps of length $g$ are at least $g$ apart. Hence $2m-1$ consecutive sites
contain at most $\nu:=\lfloor(2g+4)/g\rfloor+1\le5$ such starts, and $6m-5$ consecutive sites contain at most
$13$. Since $V_j=S_j-(2m-1)a_i$ with $0\le S_j\le\nu$ and $(2m-1)a_i\le\frac{11}9$
(Corollary~\ref{cor:pos}), we get $|V_j|\le\nu$. Using $\E|X_j|=2a_i(1-a_i)\le2a_i$, it follows that
$\sum_j\E|X_j|V_j^2\le2\nu^2Ka_i$.

The variable $Q_j=X_jV_j$ depends on $3m-2$ consecutive priorities, so $\Cov(Q_j,Q_k)=0$ unless
$k\in A_j=\{k:d(j,k)\le3m-3\}$, a set of at most $6m-5$ consecutive sites. On $A_j$ we have
$\sum_k|X_k|\le13+(6i+19)a_i$ and $\sum_k\E|X_k|\le2(6i+19)a_i$. Together with $|Q_j|\le\nu|X_j|$
and Corollary~\ref{cor:pos}(b), this gives
\[
\sum_{k\in A_j}|\Cov(Q_j,Q_k)|\le\nu^2\Bigl(\E\Bigl[|X_j|\sum_k|X_k|\Bigr]+\E|X_j|\sum_k\E|X_k|\Bigr)
\le2\nu^2a_i\Bigl(13+\frac{31}3\Bigr)\le47\nu^2a_i ,
\]
and so $\Var\sum_jQ_j\le47\nu^2Ka_i$. With $\sigma^2=b_iK\ge\frac23a_iK$, Lemma~\ref{lem:stein}
gives
\[
\dW(W_{i,K},Z)\le\frac{1}{\sqrt{a_iK}}\Bigl(2\nu^2\bigl(\tfrac32\bigr)^{3/2}+\sqrt{\tfrac{94}\pi}\,\nu\cdot\tfrac32\Bigr)
\le\frac{133}{\sqrt{a_iK}}.
\]
The Kolmogorov bound follows from $\dK(W,Z)\le(2/\pi)^{1/4}\dW(W,Z)^{1/2}$~\cite[Prop.~1.2]{Ros11}.
\end{proof}

\begin{proof}[Proof of Theorem~\ref{thm:gauss}]
Part~(a) is immediate from Theorem~\ref{thm:wass}. For~(b), recall that weak convergence in
$\R^\N$ is equivalent to convergence of the finite-dimensional
distributions~\cite[Example~2.4]{Bil99}. Fix $r\ge1$ and $\theta\in\R^r$, and let $K\ge2r+7$. Then the
first $r$ coordinates of $\mathbf V_K$ are $(D_{i,K}-a_iK)/\sqrt{b_iK}$. By
Proposition~\ref{prop:moments} the variance of $\sum_{i\le r}\theta_i\mathbf V_{K,i}$ equals
$\theta^\top\Sigma^{(r)}\theta$ \emph{exactly}, where $\Sigma^{(r)}=(\Sigma_{ij})_{i,j\le r}$.

If $\theta^\top\Sigma^{(r)}\theta=0$, the combination vanishes almost surely. Otherwise write it as
$\sum_jX_j$ with
$X_j=\sum_{i\le r}\theta_i(Y_j^{(i+1)}-a_i)/\sqrt{b_iK}$. These variables are bounded by $C/\sqrt K$
and have dependency neighbourhoods of size at most $2r+7$. Lemma~\ref{lem:stein} then gives a
Wasserstein distance $O(K^{-1/2})$ to $N(0,\theta^\top\Sigma^{(r)}\theta)$. By the Cram\'er--Wold device,
the first $r$ coordinates converge to $N(0,\Sigma^{(r)})$. The matrices $\Sigma^{(r)}$ are
consistent, positive semidefinite and have unit diagonal, so they define a centered Gaussian
law on $\R^\N$, which is the limit.
\end{proof}

For fixed $i$, the Berry--Esseen bounds for bounded summands with a dependency
graph~\cite{Rin94,CS04} give $\dK(W_{i,K},Z)=O_i(K^{-1/2})$. To obtain a local limit theorem we
exhibit a Bernoulli component in $D_{i,K}$.

\begin{proposition}\label{prop:bern}
Let $g=i+1$, $B=2g+4$, $n_0=\lfloor K/B\rfloor$ and $p_*=2\cdot4^{-4}(3/4)^{2g}/(2g)!$. There is a
$\sigma$-field $\mathcal G$ such that
\[
\bigl|\E e^{\mathrm i tD_{i,K}}\bigr|\le\exp\Bigl(-\frac{n_0\,p_*}{2\pi^2}\,t^2\Bigr)\quad(|t|\le\pi),
\qquad \Var D_{i,K}\ge\E\Var(D_{i,K}\mid\mathcal G)=\frac{n_0\,p_*}4 .
\]
\end{proposition}

\begin{proof}
Choose $n_0$ disjoint blocks of $B$ consecutive sites, and label the sites of a block
$-1,0,1,\dots,2g+1,2g+2$. The sites $-1$ and $2g+2$ are guards, $0$ and $2g+1$ are anchors, and
$1,\dots,2g$ form the interior. Call a block \emph{special} if three conditions hold: both
anchors have priority $<\frac14$, both guards have priority $>\frac34$, and the interior priorities
exceed $\frac14$ and appear in one of two orders. The first order $\pi_0$ is $U_1<U_2<\dots<U_{2g}$. The second order $\pi_1$ is
$U_g<U_1<\dots<U_{g-1}<U_{2g}<U_{2g-1}<\dots<U_{g+1}$. It has its minimum at position $g$,
increases on $1,\dots,g-1$ and decreases on $g+1,\dots,2g$. Each fixed order has conditional
probability $1/(2g)!$ given that the interior priorities exceed $\frac14$. So blocks are special
independently, each with probability $p_*$.

In a special block the anchors are valleys and the guards are not. Sites $1$ and $2g$ are never
valleys. Under $\pi_0$ the interior contains no valley; under $\pi_1$ its only valley is $g$.
Since $g\ge2$, the gaps inside $[0,2g+1]$ therefore number $H_b=0$ under $\pi_0$ and $H_b=1$
under $\pi_1$, because the gap lengths are $g$ and $g+1$. A valley strictly inside a gap is
impossible, so no gap crosses an anchor. Moreover, every site outside the special interiors has
its valley status determined by priorities outside those interiors, the anchors themselves
being forced valleys.

Let $I$ be the (random) union of the interiors of the special blocks, and let $\mathcal G$ be
generated by three things: the variables $U_x\one\{x\notin I\}$, $x\in\Z_K$; the indicators of
specialness of the blocks; and, for each special block, the unordered set of its interior
priorities. Then $D_{i,K}=C+\sum_{b\text{ special}}\one\{\text{block $b$ has order }\pi_1\}$, with $C$
$\mathcal G$-measurable. By exchangeability, given $\mathcal G$ the orders of the special blocks are
independent and equal to $\pi_0$ or $\pi_1$ with probability $\frac12$ each. Hence
$|\E[e^{\mathrm itD}\mid\mathcal G]|=|\cos(t/2)|^{N_{\rm sp}}$ and
$\Var(D\mid\mathcal G)=N_{\rm sp}/4$, where $N_{\rm sp}\sim\mathrm{Bin}(n_0,p_*)$ is the number of
special blocks. Then
$|\E e^{\mathrm itD}|\le(1-p_*(1-|\cos\frac t2|))^{n_0}\le\exp(-n_0p_*(1-\cos\tfrac t2))$, and
$1-\cos\frac t2\ge t^2/(2\pi^2)$ for $|t|\le\pi$.
\end{proof}

\begin{theorem}[local limit theorem]\label{thm:llt}
Fix $i\ge1$, let $\sigma=\sqrt{b_iK}$ and $x_k=(k-a_iK)/\sigma$. Then
\[
\sup_{k\in\Z}\Bigl|\sigma\,\Pp(D_{i,K}=k)-\varphi(x_k)\Bigr|=O_i\Bigl(\frac{\log K}{\sqrt K}\Bigr).
\]
\end{theorem}

\begin{proof}
Let $\psi(u)=\E e^{\mathrm iuW_{i,K}}$. Fourier inversion gives
\[
\sigma\Pp(D=k)-\varphi(x_k)=\frac1{2\pi}\Bigl(\int_{|u|\le\pi\sigma}e^{-\mathrm iux_k}\psi(u)\,du
-\int_\R e^{-\mathrm iux_k}e^{-u^2/2}\,du\Bigr).
\]
Split both integrals at $|u|=A\le\pi\sigma$. For $|u|\le A$ we use
$|\psi(u)-e^{-u^2/2}|\le\sqrt2\,|u|\,\dW(W_{i,K},Z)$, which holds because the real and imaginary
parts of $x\mapsto e^{\mathrm iux}$ are $|u|$-Lipschitz, together with Theorem~\ref{thm:wass}. For
$A<|u|\le\pi\sigma$, Proposition~\ref{prop:bern} gives $|\psi(u)|\le e^{-\kappa_iu^2}$ with
$\kappa_i=p_*/(4\pi^2Bb_i)$, once $K\ge2B$. The variance bound in Proposition~\ref{prop:bern}
gives $b_i\ge p_*/(4B)$, so $\kappa_i\le\pi^{-2}<\frac12$. Hence the error is at most
\[
\frac1{2\pi}\Bigl(\sqrt2A^2\,\frac{133}{\sqrt{a_iK}}+\frac{e^{-\kappa_iA^2}}{\kappa_iA}+\frac{2e^{-A^2/2}}{A}\Bigr).
\]
Choosing $A=\sqrt{\log K/\kappa_i}$ gives the claim.
\end{proof}

\begin{remark}
The Wasserstein bound depends on $i$ only through $\E D_{i,K}=a_iK$, which is the natural scale
in view of Section~\ref{sec:poisson}. The constants in Theorem~\ref{thm:llt} grow like
$(2g)!$ and are certainly not optimal. A Kolmogorov bound of order $(a_iK)^{-1/2}$ uniform in
$i$, and a local limit theorem with the correct dependence on $i$, remain open.
\end{remark}

%=========================================================================
\section{Growing gap length}\label{sec:poisson}

\begin{theorem}\label{thm:poisson}
For all $i\ge1$ and $K\ge2i+7$, with $\lambda=a_iK$,
\[
\dTV\bigl(D_{i,K},\Poi(\lambda)\bigr)\le\min(1,\lambda)\,\varepsilon_i,\qquad
\varepsilon_i=(2i+7)a_i+\frac{2(u_i+v_i)}{a_i}.
\]
Moreover $\varepsilon_i\sim\frac52ia_i$ as $i\to\infty$.
\end{theorem}

\begin{proof}
We use the Chen--Stein bound for dependency neighbourhoods~\cite[Thm.~4.7]{Ros11}
(see also~\cite{AGG89}),
$\dTV(D,\Poi(\lambda))\le\min(1,\lambda^{-1})(\Delta_1+\Delta_2)$, with the neighbourhoods $\mathcal N_j$ of the
proof of Theorem~\ref{thm:wass}. Here $\Delta_1=\sum_j\sum_{k\in\mathcal N_j}a_i^2=\lambda(2i+7)a_i$. By
Lemma~\ref{lem:local},
$\Delta_2=\sum_j\sum_{k\in\mathcal N_j\setminus\{j\}}\E Y_jY_k=2K(u_i+v_i)$. Hence
$(\Delta_1+\Delta_2)/\lambda=\varepsilon_i$. The asymptotics follow from Theorem~\ref{thm:b}.
\end{proof}

For instance $\varepsilon_6\approx0.054$, $\varepsilon_8\approx3.1\cdot10^{-3}$ and
$\varepsilon_{10}\approx1.2\cdot10^{-4}$.

\begin{corollary}\label{cor:class}
Let $1\le i(K)\le K-4$ and $\lambda_K=a_{i(K)}K$.
\begin{enumerate}[label=(\alph*)]
\item If $\lambda_K\to0$, then $\Pp(D_{i(K),K}\ne0)\le\lambda_K\to0$.
\item If $\lambda_K\to\lambda\in(0,\infty)$, then $i(K)\to\infty$ and
$\dTV(D_{i(K),K},\Poi(\lambda))\to0$.
\item If $\lambda_K\to\infty$, then $\dW(W_{i(K),K},Z)\le133/\sqrt{\lambda_K}\to0$. If in addition
$i(K)\to\infty$, then also $(D_{i(K),K}-\lambda_K)/\sqrt{\lambda_K}\Rightarrow N(0,1)$.
\end{enumerate}
Since $\log(1/a_i)=(1+o(1))\,i\log i$, case~(b) requires $i(K)=(1+o(1))\log K/\log\log K$.
Since $a_{i+1}/a_i\to0$, for every $\delta>0$ and all large $K$ at most one $i$ satisfies
$a_iK\in[\delta,1/\delta]$.
\end{corollary}

\begin{proof}
(a) follows from Proposition~\ref{prop:moments}(a) and Markov's inequality.

In~(c), eventually $K\ge2i+7$, since otherwise $\lambda_K<(2i+7)a_i\le\frac{11}9$ by
Corollary~\ref{cor:pos}. Then apply Theorem~\ref{thm:wass}; the second statement holds because
$b_i/a_i\to1$.

In~(b), suppose $K<2i(K)+7$ along a subsequence. Then $i(K)\to\infty$ along it, and
$\lambda_K<(2i+7)a_i\to0$, a contradiction. If $i(K)$ stayed bounded along a subsequence, then
$\lambda_K\to\infty$ along it, also a contradiction. Hence $i(K)\to\infty$ and eventually
$K\ge 2i(K)+7$, so Theorem~\ref{thm:poisson} applies, and finally
$\dTV(\Poi(\lambda_K),\Poi(\lambda))\to0$.
\end{proof}

%=========================================================================
\section{Further results}\label{sec:further}

\subsection{The number of roots}
The reduction of Section~\ref{sec:valleys} also gives short proofs of the results
of~\cite{MRR19} on $R^{[K]}=\sum_x\xi_x$, and a little more. Since $R^{[K]}$ is the number of local
minima of an i.i.d.\ field on the cycle graph, Proposition~\ref{prop:roots} is the case of the cycle
in the classical results on local extremes of random functions on graphs~\cite{BRS89,BR89,BRS89n}
(the $4$-cycle is the hypercube $H_2$); we include the short proof.

\begin{proposition}\label{prop:roots}
$\Var R^{[K]}=2K/45$ for every $K\ge5$, whereas $R^{[4]}=1+\mathrm{Bernoulli}(1/3)$, so that
$\Var R^{[4]}=2/9$.
\end{proposition}

\begin{proof}
For $K\ge5$ the residues $0,\pm1,\pm2$ are distinct, and $\xi_0$ is independent of $\xi_h$ for
$d(0,h)\ge3$. We have $\E\xi_0=\frac13$ and $\xi_0\xi_{\pm1}=0$. Moreover
$\E\xi_0\xi_2=\frac{16}{120}=\frac2{15}$, since sixteen of the $120$ orders of five values have
up/down word $\D\U\D\U$. Hence
$\Var R^{[K]}=K\bigl(\frac13-\frac19-\frac29+2(\frac2{15}-\frac19)\bigr)=\frac{2K}{45}$.

For $K=4$, the first-visited site is a root, and the opposite site is a root iff it precedes
both of its neighbours, which share their neighbours with the first site. This has
probability $\frac13$.
\end{proof}

\begin{remark}
Theorem~2.1 of~\cite{MRR19} states $\Var R^{[K]}=2K/45$ for $K\ge7$ and lists $K=5,6$ separately;
the values given there coincide with $2K/45$. The value $\frac19$ printed there for $K=4$
is a misprint. Indeed $R^{[4]}\in\{1,2\}$ together with $\E R^{[4]}=\frac43$ forces
$\Var R^{[4]}=\frac29$, consistent with $\Var D_{1,4}=\frac89$ in~\cite[Thm.~3.3]{MRR19} and with the
variance $2^{n-1}(n-1)/(n+1)^2$ of the number of local maxima on the hypercube $H_n$ given
in~\cite{BRS89n}, at $n=2$; the value $\frac29$ is also checked in Lean.
\end{remark}

\begin{proposition}\label{prop:formation}
Let $R_n$ be the number of roots formed by step $n$, and let $n_K^*$ be the step at which the last
root is formed. Then $\E R_n=\frac K3\bigl(1-(1-\frac3K)^n\bigr)$ for all $n\ge0$, and
\[
\Pp\Bigl(\frac{3n_K^*}K-\log\frac K3\le x\Bigr)\longrightarrow e^{-e^{-x}}\qquad(x\in\R).
\]
\end{proposition}

\begin{proof}
By Lemma~\ref{lem:root}, site $0$ is a root by step $n$ iff some step $k\le n$ selects $0$ while
steps $1,\dots,k-1$ avoid $\{-1,0,1\}$; summing $(1-\frac3K)^{k-1}/K$ over $k\le n$ gives the mean.

For the second claim we work in the Poissonized model of Lemma~\ref{lem:poisson}. Let $\tau^*$ be
the time at which the last root forms. Then $\{\tau^*\le t\}=\{\mathcal Z_t=0\}$, where
$\mathcal Z_t=\sum_xI_x$ and $I_x=\one\{T_x>t,\,T_x<T_{x\pm1}\}$. We have $\E I_x=e^{-3t}/3$; $I_x$ is
independent of $(I_y)_{d(x,y)\ge3}$; $I_0I_{\pm1}=0$; and $\E I_0I_2=\frac2{15}e^{-5t}$. The
Chen--Stein bound~\cite[Thm.~4.7]{Ros11} gives
$\dTV(\mathcal Z_t,\Poi(Ke^{-3t}/3))\le5K(e^{-3t}/3)^2+\frac4{15}Ke^{-5t}$. At
$t=(\log(K/3)+x)/3$ this is $O(K^{-2/3})$, so $\Pp(3\tau^*-\log(K/3)\le x)\to e^{-e^{-x}}$.

Finally $n_K^*=N(\tau^*)$, where $N=\sum_xN_x$. With probability tending to one $\tau^*\le\log K$,
and on this event Doob's inequality gives
$|N(\tau^*)-K\tau^*|\le\sup_{s\le\log K}|N(s)-Ks|=O_\Pp(\sqrt{K\log K})$. Hence
$3n_K^*/K-3\tau^*\to0$ in probability.
\end{proof}

In particular the root set is final after about $\frac K3\log K$ steps, a third of the
coupon-collector time.

\subsection{The growth speed}
In~\cite{MRR19} it is conjectured, on numerical grounds, that $\lim_nH_n(j)/n=\lim_n\max_kH_n(k)/n$
for every $j$, and that this limit is $\sim4/K$, i.e.\ that $Kv_K\to4$ as $K\to\infty$. In
continuous time this is the model of Atar, Athreya and Kang~\cite{AAK01}. They proved with
Kingman's theorem that the growth rate $C_K=\lim_t\max_xH(t,x)/t$ exists, and that
$3.21<C_K<5.35$; in discrete time $v_K=C_K/K$. Braun~\cite{Br20} studied the same rule on
arbitrary finite graphs. He notes that the bounds of~\cite{AAK01} hold for $K\ge5$ but not for
$K=4$, where $C_4=2+2/\sqrt3$, and that the maximum and the minimum grow at the same rate. He also
proves $C_G\le e\,\rho(A_G+I)$, which is $3e$ for cycles, and a central limit theorem for the
maximal height; he leaves the conjecture $C_K\to4$ open. Let $M_n=\max_xH_n(x)$ and
$L_n=\min_xH_n(x)$. Unrolling~\eqref{eq:bd} shows that $H_n(x)$ is the length of the longest
subsequence of $X_1,\dots,X_n$ that ends with $x$ and whose consecutive terms are at cyclic
distance at most one: a longest \emph{staircase} subsequence in the sense of~\cite{MRR20}
(this is also remarked in~\cite{AAK01}).

\begin{proposition}[\cite{AAK01,Br20}]\label{prop:speed}
Let $K\ge3$. There is $v_K\in[3/K,\,3e/K]$ such that almost surely $H_n(j)/n\to v_K$ for every
$j\in\Z_K$. Moreover $v_K=\inf_n\E M_n/n=\sup_n\E L_n/n$ and $v_3=1$.
\end{proposition}

For completeness we include a short proof; the argument for the equality of the limits is not
spelled out in~\cite{Br20}. For $K\ge5$ the bounds of~\cite{AAK01} are sharper.

\begin{proof}
If two profiles satisfy $h\ge c+h'$ pointwise for a constant $c$, the inequality persists after
a deposit at the same site in both; the same holds for $\le$. Comparing with the process
restarted from the flat profile at time $n$ gives $M_{n+k}\le M_n+M'_k$ and
$L_{n+k}\ge L_n+L'_k$, where $(M'_k,L'_k)$ is independent of $X_1,\dots,X_n$ and distributed as
$(M_k,L_k)$. Kingman's subadditive ergodic theorem~\cite{Kin68} gives
$M_n/n\to\inf\E M_n/n$ and $L_n/n\to\sup\E L_n/n$ almost surely.

To see that the limits agree, fix $j$, let $s_n=\lceil\sqrt n\rceil$ and let $x^*$ maximize
$H_{n-s_n}$. If the sites of a shortest path from $x^*$ to $j$ occur in order among
$X_{n-s_n+1},\dots,X_n$, then $H_n(j)\ge M_{n-s_n}$. The waiting times are independent geometric
variables with parameter $1/K$, and there are at most $K/2+1$ of them, so this fails with
probability at most $Ce^{-c\sqrt n}$. By Borel--Cantelli, $L_n\ge M_{n-s_n}$ eventually.

For the lower bound on $v_K$, follow the greedy chain that is extended whenever the selected site
is within distance one of its last element; its length is $1+\mathrm{Bin}(n-1,3/K)$,
cf.~\cite{MRR20}. For the upper bound, the expected number of staircase subsequences of length
$\ell$ is at most $K\,3^{\ell-1}\binom n\ell K^{-\ell}\le K(3en/(K\ell))^\ell$, which is summable over $n$ when
$\ell=\lceil\alpha n/K\rceil$ with $\alpha>3e$. For $K=3$ all sites are mutual neighbours, so each
deposit raises the maximum by exactly one.
\end{proof}

Exact enumeration of $\E M_n$ over all $K^n$ words yields $5v_5\le3.5087$ and
$6v_6\le3.7547$, sharper than~\cite{AAK01} for these $K$, and $4v_4\le3.2097$, consistent with
the exact value $4v_4=2+2/\sqrt3\approx3.1547$ of~\cite{Br20}; four simulations of $2\cdot10^9$ deposits each give $4v_4=3.15474\pm0.00006$. They suggest that $Kv_K$ increases with $K$ to a limit
$c_\infty\approx4.057$ (Figure~\ref{fig:speed}); rings of up to $64\,000$ sites give
$4.057\pm0.001$.

\begin{figure}[t]
\centering
\includegraphics[width=0.8\textwidth]{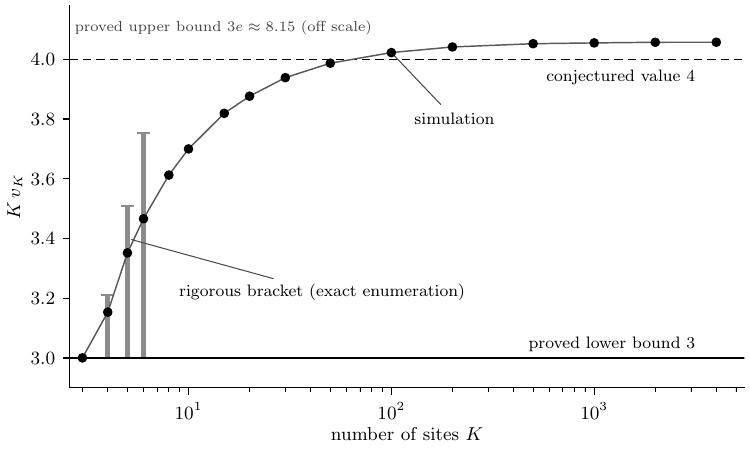}
\caption{The growth speed. Dots: simulation estimates of $Kv_K$, averaging four runs of $10^5$
deposits per site; the error bars ($\pm2$ standard deviations) are smaller than the dots. Gray
bars: rigorous brackets $3\le Kv_K\le K\,\E M_n/n$ from exact enumeration, for $K=4,5,6$ with
$n=14,11,10$; for $K=4$ the exact value $2+2/\sqrt3$~\cite{Br20} lies inside the bracket. For $K=3$,
$Kv_K=3$ exactly. The dashed line is the value $4$ suggested in~\cite{MRR19}. The upper bound $3e$
lies off the scale.}
\label{fig:speed}
\end{figure}
 If so, the
conjectured limit $4$ of~\cite{MRR19} is off by about $1.4\%$. On $\Z$ the same rule is the
next-nearest-neighbour lattice model of Penrose~\cite{Pen08} (displacement $D\equiv1$ on
$\{-1,0,1\}$), whose growth speed $\rho_1$ exists. One expects $c_\infty=\rho_1$, but we do not
prove that $Kv_K$ converges.

\subsection{Nonuniform arrival rates}
Suppose now that site $x$ is selected with probability proportional to $w_x>0$.
Lemma~\ref{lem:root} is deterministic and still applies. The proof of Lemma~\ref{lem:poisson},
with $N_x$ of rate $w_x$, shows that the roots are the valleys of the independent variables
$T_x\sim\mathrm{Exp}(w_x)$. If the rates are themselves i.i.d.\ random, then under the averaged
(annealed) law the $T_x$ are i.i.d.\ and continuous. So the root set has the same law as for
uniform rates, and Sections~\ref{sec:moments}--\ref{sec:poisson} and
Proposition~\ref{prop:roots} apply unchanged. Propositions~\ref{prop:formation}
and~\ref{prop:speed} involve the label sequence and do not transfer. For deterministic rates
the indicators $Y_j$ are no longer stationary, but the following holds.

\begin{proposition}\label{prop:rates}
Let $i\ge1$, $g=i+1$, $B=2g+4$ and $K\ge2i+7$. Suppose that $\Z_K$ contains at least $\delta K$
disjoint blocks of $B$ consecutive sites, and that $\max w/\min w\le\varrho$ within each block.
Put $\theta=1/(\varrho^{2g}(2g)!)$ and $p_\varrho=2(1-e^{-1})^2e^{-\varrho(2g+2)}\theta$. Then
\[
\kappa K\le\Var D_{i,K}\le5K,\qquad\kappa=\tfrac12\,\delta\,p_\varrho\,\theta,
\]
and
$\dW\bigl((D_{i,K}-\E D_{i,K})/\sqrt{\Var D_{i,K}},\,Z\bigr)\le(12.5\,\kappa^{-3/2}+20.4\,\kappa^{-1})\,K^{-1/2}$.
\end{proposition}

\begin{proof}
We adapt the proof of Proposition~\ref{prop:bern}, with priorities $T_x$. In a block whose rates lie
in $[c,\varrho c]$, put $\tau=1/c$. Call the block special under three conditions: both anchors
satisfy $T<\tau$; both guards and all $2g$ interior sites satisfy $T>\tau$; and the interior
order is $\pi_0$ or $\pi_1$. The first two conditions have probability at least
$(1-e^{-1})^2e^{-\varrho(2g+2)}$. Given them, the excesses $T_k-\tau$ of the interior sites are
independent with laws $\mathrm{Exp}(w_k)$, by memorylessness. Any fixed order of $2g$ such
variables has probability
$\prod_l w_{\sigma_l}/\sum_{l'\ge l}w_{\sigma_{l'}}\ge\theta$. Hence a block is special with
probability at least $p_\varrho$, independently of the other blocks. Moreover, given that it is
special, the order is $\pi_1$ with probability $q_b\in[\theta,1-\theta]$.

The combinatorics in the proof of Proposition~\ref{prop:bern} depend only on comparisons. Let
$\mathcal G'$ be generated by the variables $T_x\one\{x\notin I\}$ and the indicators of
specialness, where $I$ is the union of the special interiors. Then
$D_{i,K}=C+\sum_{b\text{ special}}\beta_b$, where $C$ is $\mathcal G'$-measurable and, given
$\mathcal G'$, the $\beta_b$ are independent Bernoulli$(q_b)$. Hence
$\Var D_{i,K}\ge\E\Var(D_{i,K}\mid\mathcal G')\ge\frac\theta2\,\E N_{\rm sp}\ge\kappa K$.

For the upper bound, let $\mathcal N_j$ be as in the proof of Theorem~\ref{thm:wass} and
$S_j=\sum_{k\in\mathcal N_j}Y_k\le\nu\le5$. Then
$\Var D_{i,K}=\sum_j\Cov(Y_j,S_j)\le\sum_j\E[Y_jS_j]\le5K$.

Finally, apply Lemma~\ref{lem:stein} with $X_j=Y_j-\E Y_j$. Since $|V_j|=|S_j-\E S_j|\le\nu$ and
$\E|X_j|\le\frac12$, we get $\sum_j\E|X_j|V_j^2\le\nu^2K/2$. As in the proof of
Theorem~\ref{thm:wass}, at most $13$ gaps start in $A_j$. Hence $\sum_{k\in A_j}|X_k|\le26$ and
$\sum_{k\in A_j}\E|X_k|\le26$, so that $\Var\sum_jX_jV_j\le26\nu^2K$. With $\sigma^2\ge\kappa K$ and
$\nu\le5$ this gives the stated bound, since $\sqrt{2\cdot650/\pi}<20.4$.
\end{proof}

%=========================================================================
\appendix
\section{Computer verification}\label{app:computations}

All finite claims were checked by exact computation in rational arithmetic.
\begin{itemize}[leftmargin=*]
\item \emph{Constants.} We computed $a_i$, $u_i$, $v_i$ and $b_i$ for $i\le30$ by counting
permutations with prescribed up/down words, using an integer dynamic program. The supplement
also contains the computer-algebra derivation of~\eqref{eq:u} and~\eqref{eq:v}. They agree with
Theorems~\ref{thm:a} and~\ref{thm:b}, and with all seven pairs $(a_i,b_i)$ of~\cite{MRR19}.
The integral representations of $u_{ij},v_{ij}$ were checked for $i,j\le5$.
\item \emph{Moments on the cycle.} We enumerated all $(K-1)!$ circular arrangements for $K\le14$
(about $6.2\cdot10^9$ for $K=14$), and, independently, all $K!$ arrangements for $K\le11$. This
confirms Proposition~\ref{prop:moments} at and above its thresholds and its failure just below
them (Remark~\ref{rem:sharp}). It also confirms $\Cov(D_{i,K},D_{j,K})=c_{ij}K$ for every pair
$i<j$ with $i+j+7\le14$, and its failure at $K=i+j+6$, as well as all values of $\E D_{1,K}$ and
$\Var D_{1,K}$, $4\le K\le8$, listed in~\cite[Thm.~3.3]{MRR19}.
\item \emph{Independent check of the constants.} The counts of permutations with prescribed descent
sets were recomputed with MacMahon's determinant formula; they agree with~\eqref{eq:u},
\eqref{eq:v} and Theorem~\ref{thm:a} for $i\le30$.
\item \emph{Root criterion.} We simulated the deposition process for $260{,}000$ random
sequences, $3\le K\le15$; the roots always agreed with Lemma~\ref{lem:root}.
\item \emph{Formalization.} The following statements were formalized and machine-checked in
Lean~4 (v4.36.0-rc1) with Mathlib, with no \texttt{sorry}; all of them use only the standard
axioms, except the value of $b_1$ in the last item:
\begin{itemize}
\item Lemma~\ref{lem:root}, and the finite form of Lemma~\ref{lem:poisson} in
Remark~\ref{rem:finite}: the order of first visits of the words visiting every site is uniform,
the roots are final once every site has been visited, and the root set has the law of the valley
set of a uniform arrangement;
\item the time-reversal duality of Remark~\ref{rem:dual};
\item Lemma~\ref{lem:local}, in full, and the locality of $Y_j$;
\item the uniformity and independence of relative orders on disjoint windows of a uniform
arrangement, and Proposition~\ref{prop:moments}\,(a) and~(b) together with~\eqref{eq:cij},
including the symmetries $u_{ij}=u_{ji}$, $v_{ij}=v_{ji}$; the moment formulas are proved both for
the uniform arrangement and directly for the root sets of the deposition process (conditionally on
every site being visited);
\item Lemma~\ref{lem:unimodal};
\item Proposition~\ref{prop:roots}, including $\Var R^{[4]}=\frac29$;
\item the valley patterns of the blocks used in Proposition~\ref{prop:bern};
\item the comparison step in the proof of Proposition~\ref{prop:speed}: domination of one profile
by another up to an additive constant survives any common sequence of deposits, which gives
$M_{n+k}\le M_n+M'_k$ and $L_{n+k}\ge L_n+L'_k$;
\item the values $a_1=\frac2{15}$ and $u_{11}=\frac{17}{315}$ (by kernel computation) and
$b_1=\frac{1772}{14175}$ (the last one using \texttt{native\_decide} for the count of $9!$
permutations).
\end{itemize}
The kernel computations of Section~\ref{sec:constants} and the probabilistic arguments of
Sections~\ref{sec:normal}--\ref{sec:further} are not formalized.
\end{itemize}
The code is available as supplementary material.

\section*{Acknowledgments}
The literature search and the Lean formalization were prepared with
the assistance of the AI system Claude (Anthropic). The author takes full responsibility for
the content of this paper.

\end{document}